%% file: finite-combinatorics.tex
\documentclass[a4paper,10pt]{article}

\makeatletter
\def\blfootnote{\gdef\@thefnmark{}\@footnotetext}
\makeatother

\usepackage{soul}

\usepackage{latexsym}
\usepackage{amsmath}
\usepackage{amscd}
\usepackage{amssymb}
\usepackage{authblk}
\usepackage[bbgreekl]{mathbbol}

\usepackage{todonotes}

\usepackage{url}
\input{macros-as-changed-by-leszek.tex}

\usepackage{xcolor}

\definecolor{lightred}{rgb}{1,.60,.60}

\begin{document}

\title{Some upper bounds on ordinal-valued Ramsey numbers \\ for colourings of pairs
}

\author[1]{Leszek Aleksander Ko{\l}odziejczyk}
\author[2]{Keita Yokoyama}

\affil[1]{\small%
Institute of Mathematics 

University of Warsaw

E-mail: \sf{lak@mimuw.edu.pl}}

\affil[2]{\small%
School of Information Science

Japan Advanced Institute of Science and Technology

E-mail:  \sf{y-keita@jaist.ac.jp}}

\date{
November 7, 2018
}

\blfootnote{%
The work of the first author
is partially supported by grant number 2017/27/B/ST1/01951 of the National Science Centre, Poland.

The work of the second author
is partially supported by
JSPS KAKENHI (grant numbers 16K17640 and 15H03634) and JSPS Core-to-Core Program
(A.~Advanced Research Networks),
 and JAIST Research Grant 2018(Houga).
 }

\maketitle

\def\Bexp{\mathrm{B}\Sigma_{1}+\mathrm{exp}}
\newcommand\RF{\mathrm{RF}}
\newcommand\tpl{\mathrm{tpl}}
\newcommand\col{\mathrm{col}}
\newcommand\fin{\mathrm{fin}}
\newcommand\Log{\mathrm{Log}}
\newcommand\Ct{\mathrm{Const}}
\newcommand\It{\mathrm{It}}
\renewcommand\PHt{\widetilde{\mathrm{PH}}{}}
\newcommand\BME{\mathrm{BME}_{*}}
\newcommand\HT{\mathrm{HT}}
\newcommand\Fin{\mathrm{Fin}}
\newcommand\FinHT{\mathrm{FinHT}}
\newcommand\wFinHT{\mathrm{wFinHT}}
\newcommand\FS{\mathrm{FS}}
\newcommand\LL{\mathsf{L}}
\newcommand\GPg{\mathrm{GP}}
\newcommand\GP{\mathrm{GP}^{2}_{2}}
\newcommand\FGPg{\mathrm{FGP}}
\newcommand\FGP{\mathrm{FGP}^{2}_{2}}
\newcommand\SGP{\mathrm{SGP}^{2}_{2}}
\newcommand\Con{\mathrm{Con}}
\newcommand\WF{\mathrm{WF}}
\newcommand\bb{\mathbf{b}}
\renewcommand\WKL{\mathrm{WKL}}
\newcommand\KL{\mathrm{KL}}
\newcommand\sWKL{\Sigma^{0}_{1}\text{-}\mathrm{WKL}}
\newcommand\ext{\mathrm{ext}}
\newcommand\VS{\mathrm{VSMALL}}
\newcommand\WO{\mathrm{WO}}

\newcommand\ome{\omega}
\newcommand\ho{\mathrm{ho}}
\newcommand\MC{\mathrm{MC}}

\begin{abstract}
We study Ramsey's theorem for pairs and two colours in the context of the theory of $\alpha$-large sets introduced by Ketonen and Solovay. We prove that any $2$-colouring of pairs from an $\omega^{300n}$-large set admits an $\omega^n$-large homogeneous set.
We explain how a formalized version of this bound gives a more direct proof, and a strengthening, of the recent result of Patey and Yokoyama [Adv.~Math.~330 (2018), 1034--1070] stating that Ramsey's theorem for pairs and two colours is $\forall\Sigma^0_2$-conservative over the axiomatic theory $\RCAo$ (recursive comprehension).
\end{abstract}

\section*{Introduction}
The work described in this paper is mostly finite combinatorics. Much of the motivation, on the other hand, comes from logic.

We contribute to the quantitative study of Ramsey's theorem for pairs in a setting where the pairs always come from a finite subset of $\N$, but
the size of the subset is given by a countable ordinal rather than just the finite ordinal specifying its cardinality. More concretely, we use the framework of $\alpha$-large sets originally due to Ketonen and Solovay \cite{KS81}, in which, for instance:
\begin{itemize}
\item a set $X \subseteq \N$ is $n$-large, for $n \in \N$, exactly if $X$ has at least $n$ elements,
\item $X$ is $\omega$-large if $X\setminus \{\min X\}$ is $\min X$-large, that is, if $X$ has strictly more than $\min X$ elements,
\item $X$ is $\omega^2$-large if $X\setminus \{\min X\}$ can be split into $\min X$ many sets $X_1,\ldots,X_{\min X}$ such that
$\max X_i < \min X_{i+1}$ and each $X_i$ is $\omega$-large,
\end{itemize} 
and so on (for precise definitions, see below). Our main aim is to obtain a good upper bound on the size of a set $X$ guaranteeing that each $2$-colouring
of $[X]^2$ will have an $\omega^n$-large homogeneous set, for $n \in \N$.

This sort of work can be viewed simply as a specific kind of finite combinatorics: essentially, the study of bounds on Ramsey numbers that happen
to take ordinal values rather than finite ones. Among the papers developing Ramsey theory in the context of $\alpha$-largeness---e.g.~\cite{BK99, BK02, BK06, Weiermann2004, KPW, DW}---many do in fact focus on the purely combinatorial side of things.  However, the original motivation for studying $\alpha$-largeness was the desire to understand the combinatorial underpinnings of (un)provability in strong axiom systems. For example, the seminal work of \cite{KS81} showed that the size of a set needed to guarantee the existence of $\omega$-large homogeneous sets for colourings of $n$-tuples grows extremely fast with $n$. This provided a combinatorial explanation for the unprovability of a statement known as the Paris-Harrington theorem in Peano Arithmetic.

Our work is also inspired by a question from logic. It follows from a general-purpose result on colourings of $n$-tuples \cite[Theorem 5]{BK02} that 
\begin{equation}\label{eqn:bigorajska-kotlarski}\omega^{\omega^n\cdot2} \to (\omega^n)^2_2.\end{equation}
That is, every $2$-colouring of pairs from an $\omega^{\omega^n\cdot2}$-large set has an $\omega^n$-large homogeneous set. It has been known that determining whether this upper bound is more or less tight would have important consequences for a longstanding open problem about the logical strength of \emph{infinite} Ramsey's theorem for pairs (see~e.g.~\cite[Question~4.4]{Seetapun1995strength} or \cite[Question~2]{montalban:questions} for the question and e.g.~\cite{CJS, CSY2014, CSY2017, BW} for some important related work). 
Recently, Patey and the second author \cite{PY} solved that open problem by showing that (\ref{eqn:bigorajska-kotlarski}) is \emph{not} tight. However, the argument in \cite{PY} was non-constructive and required a detour via infinite combinatorics and forcing; as a consequence, it did not give any specific bound.

Our main theorem here is 
\begin{equation}\label{eqn:our-bound}\omega^{300n} \to (\omega^n)^2_2.\end{equation}
This \emph{is} more or less tight, at least in the sense that it is impossible to get the left-hand side down from $\omega^{O(n)}$ to $\omega^{(1+o(1))n}$  \cite{KPW}.
Moreover, our arguments use only relatively basic finite-combinatorial tools, which means that they can be formalized in axiomatic theories of modest strength. In effect, we obtain a new, significantly more direct proof of the main result of \cite{PY}: any simple enough statement provable using infinite Ramsey's theorem for pairs and two colours can also be proved in the axiomatic theory $\RCAo$, which corresponds to a form of ``computable mathematics'' and (unlike infinite Ramsey's theorem) is too weak to imply the existence of any non-computable sets. In fact, we also obtain some improvements of that result, which provide additional information concerning 
the proof-theoretic properties of Ramsey's theorem.

The paper consists of three sections. In Section \ref{sec:general}, we provide the necessary definitions and background. In Section \ref{sec:calculation}, we prove the main theorem. Those two sections involve no logic beyond elementary facts about small infinite ordinals. The connections to logic are explained in Section \ref{sec:formalization}.

\section{$\alpha$-largeness and Ramsey $\alpha$-largeness}\label{sec:general}

We fix a primitive recursive notation for ordinals below $\ome^{\ome}$ by writing them
in \emph{Cantor normal form}: $\alpha=\sum_{i<k}\ome^{n_{i}}$ 
where $n_{i}\in\N$ and $n_{0}\ge\dots\ge n_{k-1}$.

Let $\alpha=\sum_{i<k}\ome^{n_{i}}$ and $\beta=\sum_{i<k'}\ome^{m_{i}}$. We write $\beta\unrhd\alpha$ if $m_{k'-1}\ge n_{0}$.
If $\beta\unrhd\alpha$, we can define the sum of $\beta$ and $\alpha$ as $\beta+\alpha=\sum_{i<k+k'}\ome^{t_{i}}$ where $t_{i}=m_{i}$ for $i<k'$ and $t_{j+k'}=n_{j}$ for $j<k$. In what follows, we only consider sums of this form.
We let $\beta>\alpha$ if there is $i\le k, k'$ such that $n_{j}=m_{j}$ for any $j<i$ and ($n_{i}<m_{i}$ or $i=k<k'$).
By definition, $\beta\unrhd\alpha$ implies $\beta\ge \alpha$.

We write $1$ for $\ome^{0}$, and $\ome^{n}\cdot k$ for $\sum_{i<k}\ome^{n}$.
With this notation, one can write $\alpha<\ome^{\ome}$ as $\alpha=\ome^{n}\cdot k_{n}+\dots+\ome^{0}\cdot k_{0}$, and put $\MC(\alpha)=\max\{k_{n},\dots,k_{0}\}$ ($\MC$ stands for the \emph{maximal coefficient} of $\alpha$).

For a given $\alpha<\ome^{\ome}$ and $m\in \N$, define 
$0[m]=0$, 
$\alpha[m]=\beta$ if $\alpha=\beta+1$, 
and $\alpha[m]=\beta+\ome^{n-1}\cdot m$ if $\alpha=\beta+\ome^{n}$ for some $n\ge 1$.
By definition, $m\le n$ implies $\alpha[m]\le \alpha[n]$.

The following definition combines a fundamental concept from \cite{KS81} with a variant from \cite{PY}. 

\begin{definition}[largeness]
Let $\alpha<\ome^{\ome}$, and let $n,k,m\in\N$.
\begin{enumerate}
 \item A set $X = \{ x_0 < \dots < x_{\ell-1} \}\subseteq_{\fin}\N$ 
is said to be \emph{$\alpha$-large} if $\alpha[x_{0}]\dots[x_{\ell-1}]=0$. 
In other words, any finite set is $0$-large, and $X$ is said to be $\alpha$-large if
\begin{itemize}
 \item $X\setminus \{\min X\}$ is $\beta$-large if $\alpha=\beta+1$,
 \item $X\setminus \{\min X\}$ is $(\beta+\ome^{n-1}\cdot\min X)$-large if $\alpha=\beta+\ome^{n}$.
\end{itemize}
 \item A set $X \subseteq_{\fin}\N$ 
is said to be \emph{$\RT^{n}_{k}$-$\alpha$-large} if for any $P:[X]^{n}\to k$, there exists $Y\subseteq X$ such that $Y$ is $P$-homogeneous and $\alpha$-large.
\end{enumerate}
\end{definition}


The above definition of $\omega^{n}$-largeness causes minor issues if $\min X$ is a very small number -- for instance, the set $\{0\}$ ends up being $\omega^n$-large for every $n$.
To avoid this and simplify the notation, we will always consider finite sets $X\subseteq_{\fin}\N$ satisfying $\min X\ge 3$.
We will first check several basic properties.

\begin{lem}\label{lem:alpha-decreasing}
Let $\alpha,\beta<\ome^{\ome}$ and $m\in\N$.
If $\alpha\le\beta$ and $\MC(\alpha)< m$, then $\alpha[m]\le \beta[m]$.
\end{lem}
\begin{proof}
The case $\alpha=\beta$ is trivial, so we assume $\alpha<\beta$.
Write $\beta=\beta'+\ome^{n}$.
If $\alpha\le \beta'$, then $\alpha[m]\le\beta'\le \beta[m]$.
Otherwise, $n\ge 1$ and there exists $\gamma\unlhd\beta'$ such that $\alpha=\beta'+\gamma$ and $\gamma<\ome^{n}$.
Since $\MC(\alpha)< m$, we also have $\MC(\gamma)<m$,  thus $\gamma<\omega^{n-1}\cdot m = \ome^{n}[m]$.
Therefore, we obtain $\alpha[m]\le \alpha < \beta'+\ome^{n}[m] = \beta[m]$.
\end{proof}

\begin{lem}\label{lem:alpha-large-subset}
Let $\alpha<\ome^{\ome}$ and $X,Y\subseteq_{\fin}\N$ where $X = \{ x_0 < \dots < x_{\ell-1} \}$, $Y = \{ y_0 < \dots < y_{\ell'-1} \}$
for $\ell\le \ell'$.
Assume that $y_{i}\le x_{i}$ for each $i<\ell$ and that $X$ is $\alpha$-large. Then $Y$ is $\alpha$-large.

In particular, if $X$ is $\alpha$-large and $X\subseteq Y$, then $Y$ is $\alpha$-large.
\end{lem}
\begin{proof}
We will show the following by induction on $i$:
\begin{itemize}
 \item[] for any $i<\ell$, there exists $j_{i}<\ell$ such that $j_{i}\ge i$ and $\alpha[y_{0}]\dots[y_{i}]=\alpha[x_{0}]\dots[x_{j_{i}}]$.
\end{itemize}
The base case, which corresponds to $i = -1$, is the trivial statement $\alpha=\alpha$.

Assume $\beta:=\alpha[y_{0}]\dots[y_{i}]=\alpha[x_{0}]\dots[x_{j_{i}}]$ and $i+1<\ell$.
If $\beta=0$, put $j_{i+1}=\max\{ j_{i},i+1\}$.
If $\beta=\beta'+1$, then $\beta[y_{i+1}]=\beta[x_{j_{i}+1}]$, so put $j_{i+1}=j_{i}+1$. Note that $x_{j_{i}+1}$ must exist,
because $\alpha[x_{0}]\dots[x_{j_{i}}] = \beta \neq 0 =  \alpha[x_{0}]\dots[x_{\ell-1}]$.

If $\beta=\beta'+\ome^{n}$ for some $n\ge 1$, then \[\beta[x_{j_{i}+1}]=\beta'+\ome^{n-1}\cdot (x_{j_{i}+1})=\beta[y_{i+1}]+\ome^{n-1}\cdot (x_{j_{i}+1}-y_{i+1}).\]
Since $\beta[x_{j_{i}+1}]\dots[x_{\ell-1}]=0$, we have $\ome^{n-1}\cdot (x_{j_{i}+1}-y_{i+1})[x_{j_{i}+2}]\dots[x_{\ell-1}]=0$.
(Otherwise, $\beta[x_{j_{i}+1}]\dots[x_{\ell-1}]=\beta[y_{i+1}]+\ome^{n-1}\cdot (x_{j_{i}+1}-y_{i+1})[x_{j_{i}+2}]\dots[x_{\ell-1}]>0$.)
Let $j_{i+1}$ be the smallest $j$ such that $\ome^{n-1}\cdot (x_{j_{i}+1}-y_{i+1})[x_{j_{i}+2}]\dots[x_{j}]=0$. We then have $\beta[x_{j_{i}+1}]\dots[x_{j_{i+1}}]=\beta[y_{i+1}]$.

Now, since $j_{\ell-1}$ must equal $\ell-1$, we have $\alpha[y_{0}]\dots[y_{\ell-1}]=\alpha[x_{0}]\dots[x_{\ell-1}]=0$.
\end{proof}

For a given $\alpha$-large set $X = \{ x_0 < \dots < x_{\ell-1} \}\subseteq_{\fin} \N$, take the minimum $i<\ell$ such that $\alpha[x_{0}]\cdots[x_{i}]=0$ and define $X\rest \alpha$ to be the set $\{x_{0},\dots,x_{i}\}$. (Thus, $X\rest \alpha$ is the smallest $\alpha$-large initial segment of $X$.)
\begin{lem}
\label{lem:alpha-large-partition}
Let $\alpha=\alpha_{k-1}+\dots+\alpha_{0}<\ome^{\ome}$ where $\alpha_{k-1}\unrhd\dots\unrhd\alpha_{0}$.
Then, a set $X\subseteq_{\fin}\N$ is $\alpha$-large if and only if there is a partition $X=X_{0}\sqcup\dots\sqcup X_{k-1}$ such that $\max X_{i}<\min X_{i+1}$ and $X_{i}$ is $\alpha_{i}$-large.
\end{lem}
\begin{proof}
Let $X= \{ x_0 < \dots < x_{\ell-1} \}$ be $\alpha$-large. By Lemma
\ref{lem:alpha-large-subset}, 
we can assume without loss of generality that $X = X\rest\alpha$. For each $i < k$, let $X_i$ be $(X\rest(\alpha_{i} + \ldots + \alpha_0)) \setminus (X\rest(\alpha_{i-1} + \ldots + \alpha_0))$. One checks by induction on $i$ that $X_i$ equals $(X\setminus (X_{0}\cup\dots\cup X_{i-1}))\rest \alpha_{i}$. It follows that $\max X_{i}<\min X_{i+1}$ and $X_{i}$ is $\alpha_{i}$-large.

Conversely, if $X=X_{0}\sqcup\dots\sqcup X_{k-1}$ such that $\max X_{i}<\min X_{i+1}$ and $X_{i}$ is $\alpha_{i}$-large, put $Y_{i}=X_{i}\rest \alpha_{i}$.
Then, $Y=Y_{0}\sqcup\dots\sqcup Y_{k-1}$ is $\alpha$-large by the definition, and thus $X$ is $\alpha$-large by Lemma~\ref{lem:alpha-large-subset}.
\end{proof}


In \cite{KS81}, Ketonen and Solovay use $\alpha$-largeness to analyze the Ramsey-theoretic statement known as the Paris-Harrington principle
and to clarify the relationship between the principle and hierarchies of fast growing functions.
In the process, they prove the following result 
concerning $\RT^{2}_{k}$-$\omega$-largeness.
\begin{thm}[Ketonen-Solovay \cite{KS81}, Lemma 6.4]\label{KS-theorem}
Let $n\ge 2$.
If $X\subseteq_{\fin}\N$ is $\ome^{n+4}$-large and $\min X \ge 3$, then it is $\RT^{2}_{n}$-$\omega$-large.
\end{thm}
We will give a new proof of this theorem in Subsection~\ref{subsec:ket-sol}.

Theorem \ref{KS-theorem} and its generalization to $\RT^{m}_{k}$ proved in \cite{KS81} only deal 
with the question how much $\alpha$-largeness is guaranteed to imply $\RT^{m}_{k}$-$\omega$-largeness, 
that is, the existence of an $\omega$-large homogeneous set for any given colouring. 
Our target is a generalization of the case $m=k=2$ to bounds implying $\RT^{2}_{2}$-$\omega^n$-largeness for larger $n \in \N$.
As already mentioned, even though this sort of work is purely combinatorial,
much of the motivation comes from the study of the proof-theoretic strength of infinite Ramsey's theorem for pairs. 
We discuss this in more detail in Section~\ref{sec:formalization}.

Our main result is as follows.
 
\begin{thm}\label{thm:RT22-largeness-main}
If $X\subseteq_{\fin}\N$ is $\ome^{300n}$-large and $\min X\ge 3$, then $X$ is $\RT^{2}_{2}$-$\omega^{n}$-large.
\end{thm}

\section{Calculation}\label{sec:calculation}
In this section, we prove Theorem~\ref{thm:RT22-largeness-main}.

To simplify our calculations, we only consider ``sparse enough'' finite sets.
A set $X$ with $\min X\ge 3$ is said to be \emph{exp-sparse} if for any $x,y\in X$, $x<y$ implies $4^{x}<y$.
More generally, $X$ is said to be \emph{$\alpha$-sparse} if for any $x,y\in X$, $x<y$ implies that the interval $(x,y]$ is $\alpha$-large.
Trivially, any subset of an $\alpha$-sparse set is $\alpha$-sparse.
By an easy calculation, one checks that any $\ome^{3}$-sparse set is exp-sparse:
$y>2x$ whenever $(x,y]$ is $\ome$-large, $y>x2^{x}$ whenever $(x,y]$ is $\ome^{2}$-large,
and $y>2^{2^{\dots^{x}}}$ (where there are $x$ applications of the exponential function) whenever $(x,y]$ is $\ome^{3}$-large.
\begin{lem}
\label{lem:alpha-sparse}
Let $n,m\in \N$.
If $X\subseteq_{\fin}\N$ is $(\ome^{n+m}+1)$-large and $\min X\ge 3$, then there exists $Y\subseteq X$ such that $Y$ is $\ome^{n}$-large and $\ome^{m}$-sparse.
In particular, if $X\subseteq_{\fin}\N$ is $(\ome^{n+3}+1)$-large and $\min X\ge 3$, then there exists $Y\subseteq X$ such that $Y$ is $\ome^{n}$-large and $\exp$-sparse.
\end{lem}
\begin{proof}
 We will show the following slightly stronger condition by induction on $n$:
\begin{itemize}
 \item[] if $X\subseteq_{\fin}\N$ is $(\ome^{n+m}+1)$-large and $\min X\ge 3$, then there exists $Y\subseteq X\setminus \{\max X\}$ such that $Y$ is $\ome^{n}$-large and $Y\cup\{\max X\}$ is $\ome^{m}$-sparse.
\end{itemize}
For the case $n=0$, let $X$ be $(\ome^{m}+1)$-large and take $Y = \{\min X\}$. Then $Y$ is $\omega^0$-large, i.e.~$1$-large, and it follows
from Lemma \ref{lem:alpha-large-subset} and the $(\ome^{m}+1)$-largeness of $X$ that $\{\min X, \max X\}$ is $\ome^m$-sparse. 

We turn to the case $n\ge 1$. If $X$ is $(\ome^{n+m}+1)$-large, then $X\setminus\{\min X\}$ is $\ome^{n+m}$-large, thus there exist $X_{0},\dots,X_{k-1}$ such that 
\begin{itemize}
\item $X = \{\min X,\min(X\!\setminus\! \{\min X\})\}\sqcup X_{0}\sqcup\dots\sqcup X_{k-1}$, 
\item $k = \min(X\!\setminus\! \{\min X\})\ge  1+\min X$,
\item $\max X_{i}<\min X_{i+1}$,
\item each $X_{i}$ is $\ome^{n+m-1}$-large.
\end{itemize}
Put $x_{i}=\max X_{i}$. By the induction hypothesis applied to $\{x_{i}\}\cup X_{i+1}$ for $0\le i\le k-2$, 
there exist $Y_{0},\dots,Y_{k-2}$ such that $Y_{i}\subseteq \{x_{i}\}\cup X_{i+1}\!\setminus\!\{x_{i+1}\}$, $Y_{i}$ is $\ome^{n-1}$-large and $Y_{i}\cup\{x_{i+1}\}$ is $\ome^{m}$-sparse.
Now we can check that $Y=\{\min X\}\cup Y_{0}\cup\dots\cup Y_{k-2}$ is $\ome^{n}$-large and $Y\cup\{\max X\}$ is $\ome^{m}$-sparse.
\end{proof}

The following lemma means that if a large set $X$ is 2-coloured, 
we can always choose a ``majority'' colour without losing too much of its largeness.
This fact underlies most of the constructions in the core part of our proof,
as presented in Subsection \ref{subsec:grouping}. The lemma follows from the more general \cite[Theorem 1]{BK99}, 
but our proof is very simple and---crucially for our purposes---involves no use of transfinite induction. 

\begin{lem}
\label{lem:partition-large}
For each $n\in\N$, the following holds.
\begin{enumerate}
 \item If $X=Y_{0}\cup Y_{1}\subseteq_{\fin}\N$ is $\ome^{n}\cdot 2$-large and $\exp$-sparse, then $Y_{0}$ is $\ome^{n}$-large or $Y_{1}$ is $\ome^{n}$-large.
 \item If $X=Y_{0}\cup Y_{1}\subseteq_{\fin}\N$ is $\ome^{n}\cdot (4k)$-large and $\exp$-sparse, then $Y_{0}$ is $\ome^{n}\cdot k$-large or $Y_{1}$ is $\ome^{n}\cdot k$-large.
\end{enumerate}
\end{lem}

\begin{proof}
First, we show that 1.~implies 2.~for each $n\in\N$.
If $X$ is $\ome^{n}\cdot (4k)$-large, then there exists a partition $X=X_{0}\sqcup X_{1}\sqcup\dots\sqcup X_{2k-1}$ 
such that $\max X_{i}<\min X_{i+1}$ and $X_{i}$ is $\ome^{n}\cdot 2$-large.
Then, by 1., at least one of $Y_{0}\cap X_{i}$ and $Y_{1}\cap X_{i}$ is $\ome^{n}$-large for each $i<2k$.
Depending on which case happens for at least half the $i$'s, at least one of $Y_{0}\cap X$ and $Y_{1}\cap X$ 
must be $\ome^{n}\cdot k$-large.

We now show 1.,~and thus also 2.,~by induction on $n$. The case $n=0$ is trivial, so assume $n\ge 1$.
Let $X=Y_{0}\cup Y_{1}\subseteq_{\fin}\N$ be $\ome^{n}\cdot 2$-large and $\exp$-sparse.
Take a partition $X=X_{0}\sqcup X_{1}$ so that $\max X_{0}<\min X_{1}$ and $X_0,X_1$ are both $\ome^{n}$-large.
If $X_{0}\subseteq Y_{0}$ or $X_{0}\subseteq Y_{1}$, we are done.
Otherwise, there are $c_{0},c_{1}\in X_{0}$ such that $c_{0}\in Y_{0}$ and $c_{1}\in Y_{1}$.
Put $c=\max\{c_{0},c_{1}\}$.
Then, by $\exp$-sparseness, $4^{c}<\min X_{1}$, hence $X_{1}\setminus\{\min X_{1}\}$ is $\ome^{n-1}\cdot (4c)$-large.
By 2.~of the induction hypothesis, at least one of $Y_{0}\cap X_{1}$ and $Y_{1}\cap X_{1}$ is $\ome^{n-1}\cdot c$-large.
Thus, at least one of $\{c_{0}\}\cup (Y_{0}\cap X_{1})\subseteq Y_{0}$ and $\{c_{1}\}\cup (Y_{1}\cap X_{1})\subseteq Y_{1}$ is $\ome^{n}$-large.
\end{proof}

\subsection{The grouping principle}\label{subsec:grouping}
In this subsection, we consider the notion of grouping, introduced in \cite[Section 7]{PY}
as a useful tool in the analysis of Ramsey's theorem for pairs. We will obtain an upper bound 
on the largeness of a set needed to guarantee the existence of sufficiently large groupings.

\begin{definition}[grouping]
Let $\alpha, \beta<\ome^{\ome}$.
Let $X\subseteq\N$ and let $P:[X]^{2}\to 2$ be a colouring.
A finite family (sequence) of finite sets $\langle F_{i}\subseteq X: i<\ell \rangle$
is said to be an \emph{$(\alpha,\beta)$-grouping for $P$} if
\begin{enumerate}
 \item $\A i\!<\!j\!<\!\ell\, \max F_{i}<\min F_{j}$,
 \item for any $i<\ell$, $F_{i}$ is $\alpha$-large,
 \item $\{\max F_{i}: i<\ell\}$ is $\beta$-large, and,
 \item $\A i\!<\!j\!<\!\ell\, \A x,x'\!\in\! F_{i}\,\A y,y'\!\in\! F_{j}\, \left[P(x,y)=P(x',y')\right]$.
\end{enumerate}
Moreover, $\langle F_{i}\subseteq X: i<\ell \rangle$ is said to be a \emph{strong $(\alpha,\beta)$-grouping for $P$} if
the fourth condition is replaced with
\begin{enumerate}
 \item[4'.] $\E c\!<\!2\, \A i\!<\!j\!<\!\ell\, \A x\!\in\! F_{i}\,\A y\!\in\! F_{j}\, [P(x,y)=c]$.
\end{enumerate}

\end{definition}
The intuition is that each $F_i$ is a ``group'' and that the colour of a pair consisting of representatives
of two distinct groups depends only on the groups, not on the representatives.
We say that a set $X\subseteq \N$ \emph{admits an $(\alpha,\beta)$-grouping} if for any colouring $P:[X]^{2}\to 2$, 
there exists an $(\alpha,\beta)$-grouping for $P$.
Our target theorem in this subsection is the following.

\begin{thm}
\label{thm:grouping-main}
Let $n,k\in\N$.
If $X\subseteq_{\fin}\N$ is $\ome^{n+6k}$-large and $\exp$-sparse, then $X$ admits an $(\omega^{n},\omega^{k})$-grouping.
\end{thm}

To obtain a grouping, we need to stabilize the colour between elements of any two fixed groups.
We first show how to stabilize the colour between one set and each individual element of another set. 
This will have to be done both ``from below'' and ``from above''.
\begin{lem}
\label{lem1:stabilize-colouring}
Let $X\subseteq_{\fin}\N$ be $\ome^{n+1}$-large and $\exp$-sparse, and let $c\in\N$ such that $4^{c}\le \min X$.
Then, we have the following.
\begin{enumerate}
 \item For any $W\subseteq_{\fin}\N$ such that $|W|\le c$ and $\max W<\min X$ and for any colouring $P:[W\cup X]^{2}\to 2$, there exists $Y\subseteq X$ such that $Y$ is $\ome^{n}$-large and $P(w,y)=P(w,y')$ for any $w\in W$ and $y,y'\in Y$.
 \item For any $W\subseteq_{\fin}\N$ such that $|W|\le c$ and $\max X<\min W$ and for any colouring $P:[ X\cup W]^{2}\to 2$, there exists $Y\subseteq X$ such that $Y$ is $\ome^{n}$-large and $P(y,w)=P(y',w)$ for any $w\in W$ and $y,y'\in Y$.
\end{enumerate}
\end{lem}

\begin{proof}
We only show 1., as the proof of 2.~is virtually identical. 
Since $X$ is $\ome^{n+1}$-large and $4^{c}\le \min X$, we know that $X\setminus\{\min X\}$ is $\ome^{n}\cdot 4^{c}$-large.
Put $Y_{-1}=X\setminus\{\min X\}$.
Without loss of generality, we may assume that $|W|=c$, so let $\{w_{i}: i<c\}$ be an enumeration of $W$.
Construct a sequence $Y_{0}\supseteq Y_{1}\supseteq\dots\supseteq Y_{c}$ so that $Y_{i+1}$ is $\ome^{n}\cdot 4^{c-i-1}$-large and $\A y,y'\!\in\! Y_{i+1}(P(w_{i},y)=P(w_{i},y'))$.
Indeed, Lemma~\ref{lem:partition-large} guarantees that at least one of $\{y\in Y_{i}: P(w_{i},y)=0\}$ or $\{y\in Y_{i}: P(w_{i},y)=1\}$ can be chosen as $Y_{i+1}$.
Take $Y_{c}$ as the desired set $Y$.
\end{proof}

Next, we obtain a constant-length grouping.
\begin{lem}
\label{lem2:ome-n-c-grouping}
Let $X\subseteq_{\fin}\N$ be $\ome^{n+3}$-large and $\exp$-sparse, and let $d\in\N$ such that ${d}\le \min X$.
Then, $X$ admits an $(\ome^{n},d)$-grouping.
\end{lem}
\begin{proof}
Fix a colouring $P:[X]^{2}\to 2$.
We will construct an $(\ome^{n},d)$-grouping for $P$.

First, we stabilize the colour from below in the sense of Lemma \ref{lem1:stabilize-colouring}.
Since $d\le \min X$, we know that  $X\setminus \{\min X\}$ is $\ome^{n+2}\cdot d$-large.
Take a partition $X\setminus \{\min X\}=X_{0}\sqcup\dots\sqcup X_{d-1}$ so that $\max X_{i}<\min X_{i+1}$ and $X_{i}$ is $\ome^{n+2}$-large.
Put $Y_{0}=X_{0}$, and for $i\ge 1$ take $Y_{i}\subseteq X_{i}$ so that $Y_{i}$ is $\ome^{n+1}$-large and $P(x,y)=P(x,y')$ for any $x\in X\cap [0,\max X_{i-1}]$ and any $y,y'\in Y_{i}$. This can be done using Lemma~\ref{lem1:stabilize-colouring}.1.~with $W=X\cap [d,\max X_{i-1}]$ and $c = \max X_{i-1}$, 
because $4^{\max X_{i-1}}<\min X_{i}$ by the $\exp$-sparseness of $X$.
Then, $\langle Y_{i}: i<d \rangle$ is a family of $\ome^{n+1}$-large sets such that for any $0\le i<j<d$ and for any $x\in Y_{i}$, $y,y'\in Y_{j}$, we have $P(x,y)=P(x,y')$.

Now, we stabilize the colour from above.
Note that $4^{d}\le \min Y_{i}$ for each $i<d$, because $d\le \min X<\min Y_{i}$ and all $Y_i$ are subsets of $X$ which is exp-sparse.
Put $Z_{d-1}=Y_{d-1}$, and for $i< d-1$ take $Z_{i}\subseteq Y_{i}$ so that $Z_{i}$ is $\ome^{n}$-large and $P(z,x)=P(z',x)$ for any $x\in \{\min Y_{j}: i<j<d\}$ and any $z,z'\in Z_{i}$. This can be done using Lemma~\ref{lem1:stabilize-colouring}.2.~with $W=\{\min Y_{j}: i<j<d\}$ and $c = d-i-1$.
Then, $\langle Z_{i}: i<c \rangle$ is a family of $\ome^{n}$-large sets, and for any $0\le i<j<d$ and any $x,x'\in Z_{i}$, $y,y'\in Z_{j}$, we have $P(x,y)=P(x,\min Y_{j})=P(x',\min Y_{j})=P(x',y')$.
Thus, $\langle Z_{i}: i<c \rangle$ is an $(\ome^{n},c)$-grouping for $P$.
\end{proof}

By applying Lemma \ref{lem2:ome-n-c-grouping} twice, we obtain an $\omega$-length grouping.
\begin{lem}
\label{lem3:ome-n-ome-grouping}
Let $X\subseteq_{\fin}\N$ be $\ome^{n+6}$-large and $\exp$-sparse.
Then, $X$ admits an $(\ome^{n},\ome)$-grouping.
\end{lem}
\begin{proof}
Fix a colouring $P:[X]^{2}\to 2$.
By Lemma \ref{lem2:ome-n-c-grouping}, since $2\le \min X$, there is an $(\ome^{n+3},2)$-grouping $\langle Y_{0},Y_{1} \rangle$ for $P$. Again by Lemma~\ref{lem2:ome-n-c-grouping}, since $\max Y_{0}<\min Y_{1}$, there is an $(\ome^{n},\max Y_{0})$-grouping $\langle Z_{i}: i<\max Y_{0} \rangle$ for $P$ with $Z_i \subseteq Y_1$ for each $i$. 
One can easily check that $\langle Y_{0}, Z_{0},\dots,Z_{\max Y_{0}-1} \rangle$ is an $(\ome^{n},\ome)$-grouping for $P$.
\end{proof}

Finally we prove Theorem~\ref{thm:grouping-main} by using the previous lemma repeatedly.
\begin{proof}[\it Proof of Theorem~\ref{thm:grouping-main}.]
We prove the statement by induction on $k$.
The case $k=0$ is trivial, and the case $k=1$ is Lemma~\ref{lem3:ome-n-ome-grouping}.
Assume that $k\ge 2$ and let $X\subseteq_{\fin}\N$ be $\ome^{n+6k}$-large and $\exp$-sparse.
Fix a colouring $P:[X]^{2}\to 2$.
By Lemma~\ref{lem3:ome-n-ome-grouping}, there is an $(\ome^{n+6(k-1)},\omega)$-grouping $\langle Y_{i}:i\le \ell \rangle$ for $P$.
Since $\{\max Y_{i}: i\le \ell\}$ is $\ome$-large, we know that $\ell\ge \max Y_{0}$.
By the induction hypothesis, for each $1\le i\le \ell$ there is an $(\ome^{n},\ome^{k-1})$-grouping $\langle Z^{i}_{j}: j\le m_{i} \rangle$ for $P$ such that $Z^{i}_{j}\subseteq Y_{i}$ for each $j$.
Since $\{\max Z^{i}_{j}: j\le m_{i}\}$ is $\ome^{k-1}$-large for any $1\le i\le \ell$, the set $\{\max Y_{0}\}\cup\{\max Z^{i}_{j}: j\le m_{i},1\le i\le\ell\}$ is $\ome^{k}$-large.
One can check that $\langle Y_{0},Z^{1}_{0},\dots,Z^{1}_{m_{1}},\dots,Z^{\ell}_{0},\dots,Z^{\ell}_{m_{\ell}} \rangle$ is an $(\ome^{n},\ome^{k})$-grouping for $P$.
\end{proof}

\subsection{Proof of Theorem~\ref{KS-theorem}}\label{subsec:ket-sol}
In this subsection, we give a simple proof of Theorem~\ref{KS-theorem}. The proof is still based on the original idea in \cite{KS81}, but the calculation is simplified. We include the argument to make the paper more self-contained and to facilitate the discussion of axiomatic requirements in Section~\ref{sec:formalization}.

For a given $P:[X]^{2}\to n$ and $x\in X$, define the \emph{hereditarily minimal prehomogeneous (h.m.p.h.) sequence} $\sigma_{x}\in [X]^{<\N}$ as follows:
\begin{align*}
 \sigma_{x}(0)&=\min X,\\
 \sigma_{x}(i+1)&=\min \{y\in X: y>\sigma_{x}(i)\wedge \A j\le i\,P(\sigma_{x}(j),x)=P(\sigma_{x}(j),y)
\},\\
 &\mbox{stop this construction when $\sigma_{x}(i)=x$.}
\end{align*}
One can easily check the following from the definition.
\begin{itemize}
 \item For any $i<j<k<|\sigma_{x}|$, $P(\sigma_{x}(i),\sigma_{x}(j))=P(\sigma_{x}(i),\sigma_{x}(k))$.
 \item $\sigma_{x}(i)=y<x$ if and only if $\sigma_{y}=\sigma_{x}\rest_{i+1} \neq \sigma_x$. In particular, any nonempty initial segment of $\sigma_x$ has the form $\sigma_y$ for some $y <x$.
\end{itemize}
For a given colour $c<n$, let $\ho(\sigma_{x},c)=\{\sigma_{x}(i): i<|\sigma_{x}|-1\wedge P(\sigma_{x}(i),x)=c\}$. The set $\ho(\sigma_{x},c)\cup\{x\}$ is $P$-homogeneous with colour $c$.
We let $\col(\sigma_{x})=\{c<n: \ho(\sigma_{x},c)\neq\emptyset\}$.
Clearly, $\sigma_{x}\subseteq \sigma_{y}$ implies $\col(\sigma_{x})\subseteq\col(\sigma_{y})$.
{For $x \in X \setminus\{\min X\}$, we write $\sigma_{x}^{-}$ to denote the longest initial segment $\sigma_{y}\subsetneq\sigma_{x}$ such that $\col(\sigma_{y})\subsetneq\col(\sigma_{x})$. Note that this definition would not make sense for $x = \min X$, because $\col(\sigma_{\min X}) = \emptyset$.}
\begin{lem}\label{lem:hmph-properties}
Let, $n\ge 2$, $X\subseteq_{\fin}\N$ and let $P:[X]^{2}\to n$ be a colouring.
Then we have the following.
\begin{enumerate}
 \item For any $m\in\N$, $|\{x\in X: |\sigma_{x}|\le m\}|\le n^{m}$.
 \item For any $x\in X$ and $c\in \col(\sigma_{x})$, $\min \ho(\sigma_{x},c)\le \sigma_{x}^{-}(|\sigma_{x}^{-}|-1)$. 
\end{enumerate}
\end{lem}
\begin{proof}
By the definition of h.m.p.h.~sequences,  if $\sigma_{y}=\sigma_{x}^{\frown}\langle y \rangle$ and $\sigma_{z}=\sigma_{x}^{\frown}\langle z \rangle$, then $P(x,y)\neq P(x,z)$.
Thus, for any $x\in X$, there are at most $n$-many $y$'s in $X$ such that $y>x$, $\sigma_{y}\supseteq \sigma_{x}$ and $|\sigma_{y}|=|\sigma_{x}|+1$.
Hence the size of $\{x\in X: |\sigma_{x}|\le m\}$ is at most $1+ n +\dots + n^{m-1} \le n^m$, which gives 1.

For a given $x\in X$, put $y=\max \{\min \ho(\sigma_{x},c): c\in \col(\sigma_{x})\}$.
Then $\col(\sigma_{y})\subsetneq \col(\sigma_{x})$.
Thus, $\sigma_{y}\subseteq \sigma_{x}^{-}$, and we have 2.
\end{proof}

\begin{proof}[\it Proof of Theorem~\ref{KS-theorem}.]
Let $X_{0}\subseteq_{\fin}\N$ be $\ome^{n+4}$-large and $\min X_{0}\ge 3$.
Then one can find a subset $X\subseteq X_{0}$ which is $\ome^{n}+1$-large, $\ome^{3}$-sparse and such that $\min X>n$.
Indeed, $X'_{0}=X_{0}\setminus\{\min X_{0}\}$ is at least $\ome^{n+3}\cdot 3$-large. Put $X'_{1}=X'_{0}\rest \ome^{n+3}$, $X'_{2}=(X'_{0}\setminus X'_{1})\rest \ome^{n+3}$ and $X'_{3}=(X'_{0}\setminus X'_{1}\cup X'_{2})\rest \ome^{n+3}$.
Note that $|X'_{1}|>n$. By Lemma \ref{lem:alpha-sparse}, the set $\{\min X'_{2},\max X'_{2}\}$ is $\ome^{3}$-sparse. Moreover, $\{\max X'_{2}\}\cup X'_{3}$ is $\ome^{n+3}+1$-large, so it contains an $\ome^{n}$-large $\ome^{3}$-sparse subset $X''$.
We can take $X=\{\min X'_{2}\}\cup X''$ as the desired set.

Now we show that $X$ chosen as above is $\RT^{2}_{n}$-$\ome$-large by way of contradiction.
Assume that $P:[X]^{2}\to n$ is a colouring with no $\ome$-large homogeneous set.
Write $X = \{ x_0 < \dots < x_{\ell-1} \}$.
Let $\sigma_{i}:=\sigma_{x_{i}}$ be the h.m.p.h.~sequence defined by $P$ and $x_{i}$.
For each $1\le d\le n$, we say that $i<\ell$ is $d$-critical if $|\col(\sigma_{i})|=d$ and for any $j<i$, $\sigma_{i}^{-}\neq \sigma_{j}^{-}$.
For $1\le i<\ell$ and $1\le d\le n$, define an ordinal $\gamma_{i}^{d}<\ome^{n}$ as follows.
If no $j\le i$ is $d$-critical,
put $\gamma_{i}^{d}=0$.
Otherwise, take the largest $d$-critical number $j_{0}\le i$ and let $m_{i,1}^{d}=|\{k\le i: |\col(\sigma_{k})|=d\}|$, $m_{i,2}^{d}=|\{k\le i: k$ is $(d+1)$-critical$\}|$ (where $m_{i,2}^{d}=0$ for $d=n$); then put $\gamma_{i}^{d}=\ome^{n-d}\cdot (x_{j_{0}}-m_{i,1}^{d}-m_{i,2}^{d})$.
\begin{claim*}
If there is a $d$-critical number $j\le i$, then $\gamma_{i}^{d}>0$.
\end{claim*}
\noindent \emph{Proof of Claim.} Let $j_{0}\le i$ be the largest $d$-critical number $\le i$; {since $d \ge 1$, we know that $j_0 > 0$}.
Note that for any $k\le i$ such that $|\col(\sigma_{k})|=d$, we have $\sigma_{k}^{-}=\sigma_{j}^-$ for some $j \le j_{0}$ 
(if not, there would be a $d$-critical number bigger than $j_{0}$) and therefore also $\sigma_{k}^{-}=\sigma_{j}$ for some $j<j_{0}$; this implies
$\sigma_{k}^{-}(|\sigma_{k}^{-}|-1) \le x_{j_{0}-1}$. Fix $k\le i$ such that $|\col(\sigma_{k})|=d$.
Then, for any $c\in \col(\sigma_{k})$, $\min \ho(\sigma_{k},c)\le \sigma_{k}^{-}(|\sigma_{k}^{-}|-1) \le x_{j_{0}-1}$, where the first inequality follows from Lemma~\ref{lem:hmph-properties}.2.
Since $\ho(\sigma_{k},c)\cup\{x_{k}\}$ is $P$-homogeneous and thus not $\ome$-large, we have $|\ho(\sigma_{k},c)\cup\{x_{k}\}|\le x_{j_{0}-1}$, and hence $|\sigma_{k}|\le nx_{j_{0}-1}$.
Therefore, by Lemma~\ref{lem:hmph-properties}.1, we have $m_{i,1}^{d}\le n^{nx_{j_{0}-1}}$.

If $k,k'\le i$ are both $d+1$-critical, then $\sigma_{k}^{-}\neq \sigma_{k'}^{-}$ and $|\col(\sigma_{k}^{-})|=|\col(\sigma_{k'}^{-})|=d$.
Thus, $m_{i,2}^{d}\le m_{i,1}^{d}\le n^{nx_{j_{0}-1}}$.
Finally, since $X$ is $\ome^{3}$-sparse and $x_{j_{0}-1}>n$, one can easily check that $x_{j_{0}}>2n^{nx_{j_{0}-1}}\ge m_{i,1}^{d}+m_{i,2}^{d}$.
This completes the proof of the claim.

\medskip

Now, define $\gamma_{0}=\ome^{n}$ and $\gamma_{i}=\gamma_{i}^{1}+\dots+\gamma_{i}^{n}$ for $i = 1, \ldots, \ell-1$.
Note that $1$ is $1$-critical, {because $|\col(\sigma_{1})|=1$ and $\sigma_1^- = \emptyset$ while $\sigma_0^-$ does not exist.} Thus, by the Claim, $\gamma_{i}>0$ for any $i<\ell$.

For $i<\ell-1$, consider the difference between $\gamma_{i}$ and $\gamma_{i+1}$.
Let $d=|\col(\sigma_{i+1})|$. There are two cases:
\begin{itemize}
\item if $i+1$ is $d$-critical, then $\gamma_{i+1}$ is obtained from $\gamma_{i}$ by removing one $\ome^{n-(d-1)}$
and adding at most $x_{i+1}$-many $\ome^{n-d}$'s (note that {if $i > 0$, then $d > 1$ and $\gamma^{d-1}_i > 0$ because there must be a $(d-1)$-critical $j \le i$}),
\item if $i+1$ is not $d$-critical, then $\gamma_{i+1}$ is obtained from $\gamma_{i}$ simply by removing one $\ome^{n-d}$.
\end{itemize}
In either case, $\gamma_{i+1}\le \gamma_{i}[x_{i+1}]$. Note also that $\MC(\gamma_{i})<x_{i+1}$. 
This lets us check by induction that $\gamma_{i}\le \gamma_{0}[x_{1}]\dots[x_{i}]$ for any $1\le i<\ell$.
Indeed, $\gamma_{i}\le \gamma_{0}[x_{1}]\dots[x_{i}]$ and $\MC(\gamma_{i})<x_{i+1}$ 
implies that $\gamma_{i+1}\le\gamma_{i}[x_{i+1}]\le \gamma_{0}[x_{1}]\dots[x_{i}][x_{i+1}]$
by Lemma~\ref{lem:alpha-decreasing}. 
Since $\gamma_{0}=(\ome^{n}+1)[x_{0}]$, we have $0<\gamma_{i}\le (\ome^{n}+1)[x_{0}]\dots[x_{i}]$ for any $i<\ell$.
However, $(\ome^{n}+1)[x_{0}]\dots[x_{\ell-1}]=0$ since $X$ is $\ome^{n}+1$-large. {This implies $0 < \gamma_{\ell-1} \le  0$, which is a contradiction.}
\end{proof}

\subsection{Decomposition of Ramsey's theorem for pairs}
A colouring $P:[X]^{2}\to 2$ is said to be \emph{transitive} if both $P^{-1}(0)$ and $P^{-1}(1)$ are transitive relations on $X$. Here $[X]^2$ is formally understood as the set of ordered pairs from $X$ in which the second element is strictly greater than the first: in other words, for a transitive $P$, if $x<y<z$ and $P(x,y) = P(y,z)$, then $P(x,z)$ must have the same value as well.

Using this notion, $\RT^{2}_{2}$ can be decomposed as $\RT^{2}_{2}=\EM+\ADS$ where
\begin{itemize}
 \item $\EM$: for any colouring $P:[\N]^{2}\to 2$, there exists an infinite set $H\subseteq \N$ such that $P$ is~transitive on $[H]^{2}$,
 \item $\ADS$: for any transitive colouring $P:[\N]^{2}\to 2$, there exists an infinite set $H\subseteq \N$ such that $H$ is $P$-homogeneous.
\end{itemize}
$\EM$ and $\ADS$ were originally introduced as combinatorial principles about ordered graphs and linear orders, respectively;
see \cite{HS2007,BW,Lerman2013Separating}.
We consider a similar decomposition for $\RT^{2}_{2}$-$\alpha$-largeness.
\begin{definition}
Let $\alpha<\ome^{\ome}$.
\begin{enumerate}
 \item A set $X\subseteq_{\fin}\N$ is said to be $\EM$-$\alpha$-large if for any colouring $P:[X]^{2}\to 2$, there exists $Y\subseteq X$ such that $P$ is transitive on $[Y]^{2}$ and $Y$ is $\alpha$-large.
 \item A set $X\subseteq_{\fin}\N$ is said to be $\ADS$-$\alpha$-large if for any transitive colouring $P:[X]^{2}\to 2$, there exists $Y\subseteq X$ such that $Y$ is $P$-homogeneous and $Y$ is $\alpha$-large.
\end{enumerate} 
\end{definition}

We prove Theorem \ref{thm:RT22-largeness-main} by combining appropriate upper bounds for $\EM$-$\alpha$-largeness and $\ADS$-$\alpha$-largeness.

\begin{thm}
\label{thm:EM-largeness}
If $X\subseteq_{\fin}\N$ is $\ome^{36n}$-large and $\exp$-sparse, then it is $\EM$-$\omega^{n}$-large.
\end{thm}
Note that \cite[Lemma 7.2]{PY} essentially says that for every $n$ there is an $m$ such that
an $\omega^m$-large set is $\EM$-$\omega^n$-large. 
Theorem \ref{thm:EM-largeness} strengthens this by providing a concrete upper bound on $m$,
which is possible thanks to Theorem \ref{thm:grouping-main}. 
\begin{proof}
We follow the proof of \cite[Lemma 7.2]{PY}, replacing the use of \cite[Lemma 7.1]{PY}
by Theorem \ref{thm:grouping-main}.
It is enough to show that if $X$ is $\ome^{36(n-1)+6}$-large and $\exp$-sparse then it is $\EM$-$\omega^{n}$-large.
We prove this by induction on $n$.

The case $n=1$ is just a weakening of Theorem~\ref{KS-theorem}.
Assume that $n\ge 2$ and let $X\subseteq_{\fin}\N$ be $\ome^{36(n-1)+6}$-large.
Fix $P:[X]^{2}\to 2$.
By Theorem~\ref{thm:grouping-main}, there exists an $(\ome^{36(n-2)+6}, \ome^{6})$-grouping $\langle Y_{i}:i\le \ell \rangle$ for $P$.
Theorem~\ref{KS-theorem} applied to the $\ome^{6}$-large set $\{\max Y_{i}: i\le \ell\}$ gives an $(\ome^{36(n-2)+6}, \ome)$-subgrouping $\langle Y_{i_{j}}:j\le \ell' \rangle$ which is strong, {i.e.}~there is a fixed colour $c$ such that for any $x,y$ from different groups, $P(x,y)=c$.
By the induction hypothesis, for each $j\le \ell'$ there is some $Z_{j}\subseteq Y_{i_{j}}$ such that $Z_{j}$ is $\ome^{n-1}$-large and $P$ is transitive on $[Z_{j}]^{2}$.  Since $\max Z_{0}\le \max Y_{i_{0}}\le \ell'$, the set $H=\{\max Z_{0}\}\cup\bigcup_{1\le j\le \ell'}Z_{j}$ is $\ome^{n}$-large.
Moreover, by construction, $P$ is transitive on $[H]^{2}$.
\end{proof}
\begin{thm}
\label{thm:ADS-largeness}
If $X\subseteq_{\fin}\N$ is $\ome^{4n+4}$-large and $\min X\ge 3$, then it is $\ADS$-$\omega^{n}$-large.
\end{thm}
Theorem \ref{thm:ADS-largeness} is a reformulation of \cite[Lemma 4.4]{PY}. The proof below is still based on the idea of the original proof.
\begin{proof}
Let $X\subseteq_{\fin}\N$ be an $\ome^{4n+4}$-large set with $\min X\ge 3$. Assume towards a contradiction that $X$ is not $\ADS$-$\omega^{n}$-large.
Thus, there is a transitive colouring $P:[X]^{2}\to 2$ without an $\ome^{n}$-large homogeneous set.
Given $x,y\in X$ with $x<y$, we say that an interval $[x,y]$ is $(i,\alpha)$-long if $P(x,y)=i$ and there exists an $\alpha$-large set $H\subseteq[x,y]\cap X$ such that $x,y\in H$ and $H$ is $P$-homogeneous with colour $i$.
Define a new colouring $Q: [X]^{2}\to 4n$ as follows:
\begin{align*}
 Q(x,y)=
\begin{cases}
 4k & \mbox{if $[x,y]$ is $(0,\ome^{k})$-long but not $(0,\ome^{k}+1)$-long},\\
 4k+1 & \mbox{if $[x,y]$ is $(0,\ome^{k}+1)$-long but not $(0,\ome^{k+1})$-long},\\
 4k+2 & \mbox{if $[x,y]$ is $(1,\ome^{k})$-long but not $(1,\ome^{k}+1)$-long},\\
 4k+3 & \mbox{if $[x,y]$ is $(1,\ome^{k}+1)$-long but not $(1,\ome^{k+1})$-long},\\
\end{cases}
\end{align*}
where $0\le k<n$.
Since there is no $\ome^{n}$-large $P$-homogeneous set, $Q$ is well-defined.
By Theorem~\ref{KS-theorem}, there exists an $\ome$-large $Q$-homogeneous set $\bar H\subseteq X$.
Write $\bar H=\{x_{0},\dots,x_{m}\}$ where $x_{0}<\dots<x_{m}$. By $\omega$-largeness, $m\ge x_{0}$. 

We now claim that $Q(x_{0},x_{1}) \neq Q(x_{0},x_{m})$,
which will contradict the $Q$-homogeneity of $\bar H$.
The proof of the claim splits into four cases depending on $Q(x_{0},x_{1})$. Consider for instance the case where $[x_0, x_1]$, and thus each of
$[x_{i},x_{i+1}]$, is $(0,\ome^{k}+1)$-long but not $(0,\ome^{k+1})$-long. For each $i \le m-1$, let $H_i$ be the $\ome^{k}+1$-large 
$Q$-homogeneous subset of 
$[x_{i},x_{i+1}]$ whose existence follows from the assumption that $[x_{i},x_{i+1}]$ is $(0,\ome^{k}+1)$-long.
Let $H$ be $\bigcup_{i < m} H_i$. 
Note that $x_0 \in H_0$ and that $x_i = \max H_{i-1} = \min H_i$ for $1 \le i \le m-1$; in particular, $H_{i-1} \cap H_{i} \neq \emptyset$. Thus,
by the transitivity of $P$, the set $H$ is $Q$-homogeneous with colour $0$. 
Moreover, $m\ge x_{0}$ and $H=\{x_0\} \cup \bigcup_{i < m} (H_i\!\setminus\!\{x_i\})$ imply that $H$ is $\ome^{k+1}$-large. 
Hence, $[x_{0},x_{m}]$ is $(0,\ome^{k+1})$-long, which implies $Q(x_{0},x_{1}) \neq Q(x_{0},x_{m})$.
The other cases are similar or easier.
\end{proof}

\begin{proof}[\it Proof of Theorem~\ref{thm:RT22-largeness-main}.]
We show that if $X\subseteq_{\fin}\N$ is $(\ome^{(4n+4)\cdot 36+3}+1)$-large, then it is $\RT^{2}_{2}$-$\ome^{n}$-large.
Fix a colouring $P:[X]^{2}\to 2$.
First, using Lemma~\ref{lem:alpha-sparse}, take $X_{0}\subseteq X$ which is $\ome^{(4n+4)\cdot 36}$-large and $\exp$-sparse.
Next, using Theorem~\ref{thm:EM-largeness}, take $X_{1}\subseteq X_{0}$ such that $X_{1}$ is $\ome^{4n+4}$-large and $P$ is transitive on $[X]^{2}$.
Finally, Theorem~\ref{thm:ADS-largeness} gives $Y\subseteq X_{1}$ which is  $\ome^{n}$-large and $P$-homogeneous.
\end{proof}

\begin{remark}\label{rem:better-bound}
One may obtain slightly better bounds for some of the theorems/lemmas above.
For example, in Lemma \ref{lem2:ome-n-c-grouping}, if $d=2$ then we only need $X$ to be $\ome^{n+2}$-large, because we only need to shrink $X_1$ in the first stage of the proof and $Y_{0} = X_0$ in the second stage.  
This could actually  be used to obtain a slightly better upper bound ($\ome^{n+5k}$-largeness) in Theorem~\ref{thm:grouping-main} but such small improvements are not particularly important from our perspective.

On the other hand, the bound in Theorem~\ref{thm:grouping-main} cannot be reduced to $\ome^{n+o(n)}$-largeness.
Indeed, Kotlarski~et~al.~\cite[Theorem~5.4]{KPW} showed that if a set $X$ is $\RT^{2}_{2}$-$\ome^{n}$-large, then it is $\ome^{2n}$-large.
\end{remark}

\section{Finite consequences of Ramsey's theorem for pairs}\label{sec:formalization}
In this section, we explain the relevance of Theorem \ref{thm:RT22-largeness-main} to logic, or more specifically to proof theory. Ramsey-theoretic principles are well-known to display interesting behaviour with respect to provability in axiomatic theories. For instance, the already mentioned Paris-Harrington principle, 
which states: 
\begin{center}
``for every $n, \ell \in \N$ there exists a finite set $X\subseteq_{\fin}\N$ such that $\min X \ge \ell$ and $X$ is $\RT^{n}_{2}$-$\omega$-large''
\end{center}
is unprovable in Peano Arithmetic (see e.g.~\cite{Ha-Pu}). In contrast, it was recently proved in \cite{PY} that (infinite) Ramsey's theorem for pairs and two colours is in a certain sense proof-theoretically ``tame''. Theorem \ref{thm:RT22-largeness-main} makes it possible to give a more direct proof of that result and, in fact, to strengthen it. 

To understand the proofs in this section, the reader will need some familiarity with axiomatic theories of first- and second-order arithmetic and their models---see \cite{SOSOA, Ha-Pu} for details. The following very brief review will hopefully suffice for understanding the statements of the results. The \emph{language of second-order arithmetic} has two types of variables: \emph{first-order} variables $x,y,z,\ldots$ or 
$k,\ell, n,\ldots$ to stand for natural numbers (which can also be used to code other finite objects, such as finite subsets of $\N$) and \emph{second-order} variables to stand for subsets of $\N$ (which can also be used to code relations on $\N$). A formula in this language is $\Sigma^0_n$ if it has no second order quantifiers and consists of at most $n$ first-order quantifiers (beginning with $\exists$) followed by a formula in which all quantifiers have to be \emph{bounded}, i.e.~of the form $\exists x\!<\!y$ or $\forall x\!<\!y$. The dual class of formulas beginning with $\forall$ is called $\Pi^0_n$, while $\forall\Sigma^0_n$ stands for the class of formulas consisting of universal (possibly second-order) quantifiers followed by a $\Sigma^0_n$ formula. $\RCAo$ is an axiomatic theory in this language which has: (a) some basic axioms specifying that $\N$ is a discrete ordered semiring, (b) the \emph{$\Delta^0_1$-comprehension} axiom, which states that for every decidable property $R$ of natural numbers (as given by an appropriate syntax) the set $\{n\in \N: R(n)\}$ exists, and (c) the $\Sigma^0_1$-induction axiom, which allows the use of mathematical induction for any property expressed by a $\Sigma^0_1$ formula (which in fact means: for any recursively enumerable property). $\RCAo$ may be viewed as embodying the methods of ``computable mathematics''. For each $n$, any $\forall\Sigma^0_n$ statement provable in $\RCAo$ is provable in the weaker theory $\II$, which only has axioms of type (a), (c). $\EFA$ (\emph{Elementary Function Arithmetic}) is an even weaker theory in which mathematical induction can only be used for properties defined without using any unbounded quantifiers; to counteract this weakness, $\EFA$ has to include an additional axiom guaranteeing the basic properties of the exponential function on $\N$, including the fact that $2^n$ exists for every $n \in \N$.

The main result of \cite{PY} concerns the theory $\WKLo + \RT^2_2$, which is obtained by adding a weak version of K\"onig's Lemma and a natural statement of Ramsey's theorem for pairs and two colours to $\RCAo$.

\begin{theorem*}\cite[Theorem 7.4]{PY}\label{thm:py}
$\WKLo+\RT^{2}_{2}$ is $\forall\Sigma^0_2$-conservative over $\RCAo$. That is, each $\forall\Sigma^0_2$ statement provable in $\WKLo + \RT^2_2$ is already provable in $\RCAo$.
\end{theorem*}

The combinatorial core of the proof of this theorem in \cite{PY} is contained in the following result about $\alpha$-largeness. Here and below,
ordinals smaller than $\omega^\omega$ are represented in $\RCAo$ by letting the number coding $\langle n_0, \ldots, n_{k-1} \rangle$ stand for $\sum_{i<k}\ome^{n_{i}}$.

\begin{proposition*}\cite[Proposition 7.7]{PY}
For every natural number $n$ there exists a natural number $m$ such that $\RCAo$ proves: for every $X\subseteq_{\fin}\N$ with $\min X \ge 3$, if $X$ is $\omega^m$-large, then $X$ is $\RT^{2}_{2}$-$\omega^n$-large.
\end{proposition*}

However, the proof of \cite[Theorem 7.4]{PY}\label{thm:py} does not work with \cite[Proposition 7.7]{PY} directly, but instead makes use of an intermediate notion of ``density''. Moreover, even though \cite[Proposition 7.7]{PY} is a statement of finite combinatorics, its proof involves a major detour through an infinitary principle (cf.~\cite[Section 6]{PY}). Our proof of Theorem \ref{thm:RT22-largeness-main} is considerably more direct and it is readily seen to give the following stronger version of \cite[Proposition 7.7]{PY}:
\begin{corollary}\label{thm:RT22-largeness-main-formalized}
$\RCAo$ (and, in fact, the weaker theory $\EFA$) proves the following: for every $n \in \N$ and every $X\subseteq_{\fin}\N$ with $\min X\ge 3$, if $X$ is 
$\omega^{300n}$-large, then $X$ is $\RT^{2}_{2}$-$\omega^{n}$-large.
\end{corollary}
\begin{proof}
An inspection of the arguments in Sections \ref{sec:general} and \ref{sec:calculation} (including the proof of Theorem~\ref{KS-theorem} as presented in Subsection \ref{subsec:ket-sol}) reveals that they only make use of elementary manipulations of finite combinatorial objects such as finite sets, finite trees and Cantor Normal Forms, and of the usual principle of mathematical induction applied to properties that can be expressed using bounded quantifiers, possibly with exponentially large bounds. These tools are available within $\EFA$. (A different proof of Theorem~\ref{KS-theorem} in $\EFA$ was recently given by Pelupessy \cite{Pelupessy-unpublished}.)

Crucially, none of the arguments involve transfinite induction up to $\omega^\omega$ (which is not available in $\RCAo$) or mathematical induction for 
$\Sigma^0_1$ or $\Pi^0_1$ properties whose definitions require unbounded quantifiers (this would be available in $\RCAo$ but not in $\EFA$). Regarding the second point, note that all apparent uses of $\Pi^0_1$-induction---as in, for instance, the proof of Theorem \ref{thm:grouping-main}, where we seem to be using induction for a statement quantifying over all $X\subseteq_{\fin}\N$---can be replaced by bounded induction: for any given $X$, the universal quantifier in the induction property can be restricted to range over subsets of $X$.  
\end{proof}

The extra strength provided by Corollary \ref{thm:RT22-largeness-main-formalized} can be used to obtain
a strengthening of \cite[Theorem 7.4]{PY}, by means of a relatively simple proof that avoids the concept of
density. To express the strengthening, let $\WO(\alpha)$, for $\alpha < \omega^\omega$, 
denote the statement that there is no infinite descending sequence of ordinals starting from $\alpha$. 
The following lemma lists some basic properties of ordinals below $\omega^\omega$ provable within $\RCAo$. 
The properties are well-known and their easy proofs seem to be part of the folklore.

\begin{lemma}\label{lem:ordinals-in-rca}
The following are provable within $\RCAo$.
\begin{enumerate}
 \item For any $\alpha<\omega^{\omega}$, $\WO(\alpha)$ if and only every set of ordinals smaller than $\alpha$ has a minimum element.
 \item For any $\alpha<\omega^{\omega}$, $\WO(\alpha)$ if and only if any infinite set contains an $\alpha$-large subset.
 \item For any $m\in\N$, $\WO(\omega^{m})$ implies $\WO(\omega^{2m})$.
\end{enumerate}
\end{lemma}

\noindent In contrast, $\RCAo$ is unable to prove ``$\WO(\alpha)$ holds for every $\alpha < \omega^\omega$''.

\begin{theorem}\label{thm:py-strong}
$\WKLo+\RT^{2}_{2}$ is conservative over $\RCAo$ with respect to sentences of the form:
\[\A \alpha\!<\!\omega^{\omega}\,(\WO(\alpha)\to \varphi(\alpha))\]
where $\varphi$ is $\forall\Sigma^0_2$.
\end{theorem}

Note that the class of sentences considered in Theorem \ref{thm:py-strong} is strictly larger than the one in \cite[Theorem 7.4]{PY}\label{thm:py} because
$\WO(\alpha)$ is not a $\Sigma^0_3$ statement (it is in fact $\forall \Sigma^0_2$).

\begin{proof}
(In this argument, we follow the notational conventions of \cite{PY}, using the symbol $\bbomega$ to denote the smallest infinite ordinal \emph{as formalized in $\RCAo$} and reserving $\omega$ for the set of actual (standard) natural numbers.)
 
Let $\varphi(\alpha)\equiv \A X\, \A x\, \E y\, \A z\,\varphi_{0}(X[z], x,y,z,\alpha)$, where $\varphi_{0}$ is $\Sigma^{0}_{0}$, be a $\forall \Sigma^{0}_{2}$-formula such that $\RCAo$ does not prove $\A \alpha\!<\! \bbomega^{\bbomega}\,(\WO(\alpha)\to \varphi(\alpha))$.
Take a countable nonstandard model $(M,S)\models\RCAo+\exists \alpha\!<\!\bbomega^{\bbomega}\,(\WO(\alpha) \land \neg \varphi(\alpha))$.
There exist $A\in S$ and $a,\alpha\in M$ such that \[(M,S)\models\alpha<\bbomega^{\bbomega}\wedge\WO(\alpha)\wedge\A y \, \E z\, \neg\varphi_{0}(A[z],a,y,z,\alpha).\]
Take some $c\in M\setminus\{0\}$ such that $\alpha<\bbomega^{c}$ and $\WO(\bbomega^{c})$ holds in $(M,S)$.
(If $\alpha=\bbomega^{c_{0}}+\beta$, then Lemma~\ref{lem:ordinals-in-rca} part~3.~lets us take $c:=c_{0}+1$.)
Also take some $b\in M$ which is greater than each of $a$, $c$, and the code for $\alpha$.
Use primitive recursion in $(M,S)$ to define a sequence $\langle x_{i}: i\in M \rangle$ such that $x_{0}=b$ and $x_{i+1}=\min\{x>x_{i}: \A y\!<\!x_{i}\,\E z\!<\!x\,\neg\varphi_{0}(A[z],a,y,z,\alpha)\}$.
By $\Delta^0_1$-comprehension in $(M,S)$, the set $Y=\{x_{i}: i\in M\}$ belongs to $S$. Moreover, $Y$ is infinite in $(M,S)$.

By Lemma~\ref{lem:ordinals-in-rca} parts~2.~and 3., every infinite set contains an $\bbomega^{{n}c}$-large finite subset for each $n \in \omega$. It follows that $Y$ has an $\bbomega^{{n}c}$-large $M$-finite subset for each $n \in \omega$. 
By overspill, there exists an $M$-finite set $X\subseteq Y$ which is $\bbomega^{300^{d}c}$-large for some $d\in M\setminus \omega$.

Let $\{E_{i}\}_{i\in\omega}$ be an enumeration of all $M$-finite sets which are not $\bbomega^{c}$-large, and
$\{P_{i}\}_{i\in\omega}$ be an enumeration of all $M$-finite functions from $[[0,\max X]]^{2}$ 
to $2$.
We will construct an $\omega$-length sequence of $M$-finite sets 
$X=X_{0}\supseteq X_{1}\supseteq\dots$ such that for each~$i \in \omega$, the set $X_{i}$ is $\bbomega^{300^{d-i}c}$-large,
the colouring $P_{i}$ is constant on $[X_{2i+1}]^{2}$, and $[\min X_{2i+2},\max X_{2i+2})\cap E_{i}=\emptyset$.

To achieve this, we do the following for each $i\in\omega$.
At stage $2i+1$ of the construction, 
we take $X_{2i+1}\subseteq X_{2i}$ such that $P_{i}$ is constant on $[X_{2i+1}]^{2}$.
Assuming $X_{2i}$ was $\bbomega^{300^{d-2i}c}$-large, Corollary \ref{thm:RT22-largeness-main-formalized} lets us choose $X_{2i+1}$ so that it is 
$\bbomega^{300^{d-2i-1}c}$-large.
Then, at stage $2i+2$, consider the colouring $Q\colon[X_{2i+1}]^{2}\to 2$ such that $Q(x,y)=0$ if and only if $E_{i}\cap[x,y)=\emptyset$.
Again by Corollary~\ref{thm:RT22-largeness-main-formalized}, we take $X_{2i+2}\subseteq X_{2i+1}$ such that $Q$ is constant on $[X_{2i+2}]^{2}$ and $X_{2i+2}$ is $\bbomega^{300^{d-2i-2}c}$-large. $X_{2i+2}$ is in particular $(\bbomega^c+1)$-large, so if the colour of $Q$ on $[X_{2i+2}]^{2}$ was $1$, then by Lemma \ref{lem:alpha-large-subset} the set $E_i$ would be $\bbomega^c$-large.
Therefore, the colour of $Q$ on $[X_{2i+2}]^{2}$ must be $0$, which implies $[\min X_{2i+2},\max X_{2i+2})\cap E_{i}=\emptyset$.

Now, let $I=\sup\{\min X_{i}: i\in\omega\}\subseteq_{e} M$.
The even-numbered stages of our construction ensure that $I$ is a cut in $M$ 
and that $X_{j}\cap I$ is unbounded in $I$ for each $j\in\omega$ (consider the case where $E_i$ is a singleton set). 
They also ensure that that any set $E\in \Cod(M/I)$ which is unbounded in $I$ has an $\bbomega^{c}$-large subset.
To see this, assume $E$ has no $\bbomega^{c}$-large subset and take an $M$-finite set $\hat E$ such that $E=\hat E\cap I$.
By overspill, there exists $e\in M\setminus I$ such that $\hat E\cap [0,e]$ has no $\bbomega^{c}$-large subset, 
but then $\hat E\cap [0,e]=E_{i}$ for some $i\in\omega$ and so by construction $E=\hat E\cap I=E_{i}\cap I$ must be bounded in $I$.

It follows in particular that $I$ is a semi-regular cut---that is, for every $e \in I$, any $E\in \Cod(M/I)$ which is unbounded in $I$
has an $M$-finite subset with at least $e$ elements. By standard arguments, this implies $(I,\Cod(M/I))\models\WKLo$. Thus, by Lemma~\ref{lem:ordinals-in-rca} part~2., we also get $(I,\Cod(M/I))\models\WO(\bbomega^{c})$.

On the other hand, the odd-numbered stages ensure that $(I,\Cod(M/I))\models\RT^{2}_{2}$. To see this, let
$P:[I]^{2}\to 2$ be a function in $\Cod(M/I)$. Then $P=P_{i}\cap I$ for some $i\in\omega$.
Hence $P$ is constant on $[X_{2i+1}\cap I]^{2}$, and $X_{2i+1}\cap I\in\Cod(M/I)$ is an infinite set in $I$.

Finally, since $X\cap I$ is unbounded in $I$, so is $Y\cap I$. Thus, we have \[(I,\Cod(M/I))\models \A y\,\E z\,\neg\varphi_{0}((A\cap I)[z],a,y,z,\alpha),\] and hence $(I,\Cod(M/I))\models\neg\varphi(\alpha)$.
We have $(I,\Cod(M/I))\models\WO(\alpha)$ because $\alpha<\bbomega^{c}$.
Therefore, $\WKLo+\RT^{2}_{2}$ does not prove $\A \alpha\!<\!\bbomega^{\bbomega}\,(\WO(\alpha)\to \varphi(\alpha))$.
\end{proof}

The following consequence of Theorem \ref{thm:py-strong} states, intuitively speaking, that $\RT^{2}_{2}$ does not imply
any new closure properties of ordinals below $\omega^{\omega}$ compared to $\RCAo$.
\begin{corollary}
For any primitive recursive function $p\colon\omega^{\omega}\to\omega^{\omega}$ (defined on codes of ordinals), if $\RT^{2}_{2}+\WKLo$ proves 
\[ \A \alpha\!<\!\omega^{\omega}\,(\WO(\alpha)\to \WO(p(\alpha))),\]
 then $\RCAo$ proves the same statement.
\end{corollary}
As a special case, $\RT^{2}_{2}+\WKLo$ does not prove $\A x\,(\WO(\omega^{x})\to \WO(\omega^{2^{x}}))$, as this is not provable within $\RCAo$. (Note, though, that already the model constructed in the proof of Theorem~\ref{thm:py-strong} satisfies $\WO(\omega^{c}) \land \neg \WO(\omega^{2^c})$.)

Another strengthening of \cite[Theorem 7.4]{PY}\label{thm:py} -- in fact, the original motivation for Corollary \ref{thm:RT22-largeness-main-formalized} -- concerns proof lengths. Corollary \ref{thm:RT22-largeness-main-formalized} can be used to obtain the theorem below, which states that
$\WKLo + \RT^2_2$ has no significant proof speedup for proofs of $\forall\Sigma^0_2$ sentences over $\RCAo$. This answers Question 9.5 of \cite{PY} in the negative.

\begin{theorem}\label{thm:no-speedup}
There is a polynomial-time computable mapping which, given a proof $p$ of a $\forall\Sigma^0_2$ sentence $\varphi$ in $\WKLo + \RT^2_2$ as input, returns a proof $p'$ of $\varphi$ in $\RCAo$ as output. In particular, the size of $p'$ is at most polynomially larger than the size of $p$.
\end{theorem}

Proving Theorem \ref{thm:no-speedup} requires a more extensive development of the logical framework. The proof will be provided in the forthcoming paper \cite{kwy:ramsey-proofsize}.

\bibliographystyle{plain}
\bibliography{bib}

\end{document}

%% file: macros-as-changed-by-leszek.tex
\usepackage{amsthm}

\def\RCAo{\mathsf{RCA_0}}

\def\WKLo{\mathsf{WKL_0}}

\def\WKL{\mathsf{WKL_0}}

\def\EFA{\mathrm{EFA}}

\def\E{\exists}
\def\A{\forall}
\def\N{\mathbb{N}}

\def\Cod{\mathrm{Cod}}

\def\rest{{\upharpoonright}}

\def\P2{\Pi^1_2}

\def\PHt{\mathrm{PH}^2_2}
\def\RT{\mathrm{RT}}

\def\ADS{\mathrm{ADS}}

\newcommand\EM{\mathrm{EM}}

\def\II{\mathrm{I}\Sigma^0_1}

\newcounter{menum}
{\begin{enumerate}%
\setcounter{enumi}{#1}}%
{\setcounter{menum}{\value{enumi}}\end{enumerate}}

\newtheorem{thm}{Theorem}[section]
\newtheorem{theorem}[thm]{Theorem}

\newtheorem*{theorem*}{Theorem}

\newtheorem*{claim*}{Claim}

\newtheorem{lem}[thm]{Lemma}

\newtheorem{lemma}[thm]{Lemma}
\newtheorem{corollary}[thm]{Corollary}

\newtheorem*{proposition*}{Proposition}

\theoremstyle{definition}

\newtheorem{definition}[thm]{Definition}

\newtheorem*{remark}{Remark}
